\documentclass[12pt,a4paper]{amsart}
\usepackage{amssymb,latexsym,amsmath,amscd,graphicx,url,color}
\theoremstyle{plain}
\newtheorem{thm}{Theorem}[section]

\newtheorem{lemma}[thm]{Lemma}

\theoremstyle{definition}
\newtheorem{defn}[thm]{Definition}
\newtheorem{rmk}[thm]{Remark}

\newcommand{\NN}{{\mathbb N}}
\newcommand{\QQ}{{\mathbb Q}}

\newcommand{\ZZ}{{\mathbb Z}}

\DeclareMathOperator{\HS}{HS}
\DeclareMathOperator{\GL}{GL}
\DeclareMathOperator{\gin}{gin}
\DeclareMathOperator{\inid}{in}

\DeclareMathOperator{\Supp}{Supp}
\DeclareMathOperator{\Ker}{Ker}

\title{Multigraded generic initial ideals of determinantal ideals}

\author{A. Conca}
\address{Dipartimento di Matematica, 
Universit\`a di Genova, Via Dodecaneso 35, 
I-16146 Genova, Italy}
\email{conca@dima.unige.it}

\author{E. De Negri}
\address{Dipartimento di Matematica, 
Universit\`a di Genova, Via Dodecaneso 35, 
I-16146 Genova, Italy}
\email{denegri@dima.unige.it}

\author{E. Gorla}
\address{Institut de Math\'ematiques, Universit\'e de Neuch\^atel, Rue Emile-Argand 11, CH-2000
  Neuch\^atel, Switzerland}  
\email{elisa.gorla@unine.ch}

\thanks{The first and the second authors were partially supported by GNSAGA-INdAM. The third author was partially supported by the Swiss National Science
  Foundation under grant no. 200021\_150207.}

\subjclass[2010]{Primary }

\begin{document}

\maketitle 

\dedicatory{{\small{\em \begin{center}Dedicated to our friend, teacher,  and colleague Winfried Bruns\\ on the occasion of his seventieth birthday.\end{center}}}}

\begin{abstract}
Let $I$ be either the ideal of maximal minors or the ideal of $2$-minors of a row graded or column graded matrix of linear forms $L$.  In \cite{CDG1,CDG2}  we showed that  $I$ is a Cartwright-Sturmfels ideal, that is,  the multigraded generic initial ideal $\gin(I)$ of $I$ is radical (and essentially independent of the term order chosen). In this paper we describe  generators  and prime decomposition of $\gin(I)$ in terms of data  related to the linear dependences  among the row or columns of  the submatrices of $L$. In the case of $2$-minors we also give a closed formula for its multigraded Hilbert series.
\end{abstract}

\section*{Introduction}

Ideals of minors of matrices of linear forms are widely studied within commutative algebra and algebraic geometry. For example, they arise in classical constructions in invariant theory and define certain Veronese and Segre varieties. One often concentrates on matrices whose entries satisfy extra conditions, such as generic matrices (whose entries are distinct variables), generic symmetric matrices, catalecticant matrices, and 1-generic matrices. In this paper, we study ideal of minors of matrices which are homogeneous with respect to a multigrading. More precisely, given a standard $\ZZ^m$-multigraded polynomial ring we study the ideals of maximal minors or 2-minors of $m\times n$ matrices with the property that all the entries in the $i$-th row are homogeneous of degree $e_i\in \ZZ^m$. We call such matrices row graded. Similarly one can define column graded matrices. 

Gr\"obner bases of ideals of minors have been extensively studied, and a wealth of results is available on ideals of minors of generic matrices, generic symmetric matrices, and catalecticant matrices for specific term orders. In particular, in~\cite{BZ} and~\cite{SZ} Bernstein, Sturmfels, and Zelevinsky showed that the maximal minors of a matrix of variables are a universal Gr\"obner basis of the ideal that they generate. Sturmfels \cite{S} and Villarreal \cite{V} produced a universal Gr\"obner basis of the ideal of $2$-minors of a matrix of variables. These results were generalized in~\cite{AST,B,CS,C,K}. In \cite{CDG1} and~\cite{CDG2} we studied universal Gr\"obner bases of the ideals of maximal minors and $2$-minors of row or column graded matrices. In~\cite{CDG2} we introduced two new families of ideals, that we call Cartwright-Sturmfels and Cartwright-Sturmfels$^*$. The families are defined in terms of properties of their multigraded generic  initial ideals, and they are instrumental to our study of universal Gr\"obner bases. Some of the main results of~\cite{CDG1} and~\cite{CDG2} consist of showing that the ideals of minors that we study are Cartwright-Sturmfels, and that some of them are also Cartwright-Sturmfels$^*$. In particular, we showed that their multigraded generic initial ideals are radical and do not depend on the term order.  

In~\cite[Theorem~3.2 and Theorem~4.1]{CDG1} we described the monomial generators of the multigraded generic initial ideal of the ideal of maximal minors of a column graded matrix and of a row graded matrix respectively, the latter under the assumption that the ideal has maximal height. We recall these results in Sections~\ref{rowgrad} and~\ref{colgrad}.  In Theorem~\ref{ginrowgraded} we describe  a system of generators and the prime decomposition of the multigraded generic initial ideal of the ideal of maximal minors of any row graded matrix. In Theorem~\ref{lemac3}  we do the same for the ideal of 2-minors of a row graded matrix. In addition we give a closed formula for the multigraded Hilbert series of the ideals of $2$-minors.

\section{Preliminaries}\label{prelim}

Let $S$ be  a polynomial ring  over a field $K$, endowed with a standard  $\ZZ^v$-graded structure, i.e., the degree of every indeterminate is an element of the canonical basis $\{e_1,\dots,e_v\}$ of $\ZZ^v$. For $i=1,\dots,v$ let $u_i$ be the number of indeterminates of $S$ of degree $e_i$. We denote them by $x_{i1},\dots, x_{iu_i}$. We assume that $u_i>0$ for all $i$. 

The group $G=\GL_{u_1}(K)\times \cdots \times \GL_{u_v}(K)$ acts on $S$ as the  group of $\ZZ^v$-graded $K$-algebras automorphisms. Let $B=B_{u_1}(K)\times \cdots  \times B_{u_v}(K)$ be the {\em Borel subgroup} of $G$, consisting of the upper triangular invertible matrices. An ideal $I\subset S$ is {\em Borel fixed} if $g(I)=I$ for every $g\in B$. In analogy with the standard $\ZZ$-graded situation, the property of being Borel fixed can be characterized combinatorially. Indeed one has that  an ideal  $I$ of $S$  is Borel fixed with (respect to the given $\ZZ^v$-graded structure) if and only if  it satisfies the following conditions: 
\begin{itemize} 
\item[(1)] $I$ is generated by monomials,
\item[(2)]  For every monomial generator $m$ of  $I$ one has that $(x_{ik}/x_{ij})^d m\in I$ for every $i=1,\dots,v$, for every $1\leq k<j\leq u_i$ and every $0\leq d\leq c$ such that $\binom{c}{d}\neq 0$ in $K$ where $c$ is the exponent of $x_{ij}$ in $m$.
\end{itemize} 

 Given a term order $\tau$ and a $\ZZ^v$-graded homogeneous ideal $I$ of $S$, one can consider its {\em multigraded generic initial ideal} $\gin_{\tau}(I)$ defined as $\inid_\tau(g(I))$, where $g$ is a general element in $G$. 
Notice that we always assume that $$x_{ij}>x_{ik} \ \mbox{ if } 1\leq j<k\leq u_i.$$
As in the standard graded setting, multigraded generic initial ideals are Borel fixed and can be obtained as $\inid_{\tau}(b(I))$ for a general $b\in B$. 

In~\cite{CDG2}, we introduced the following two families of multigraded ideals. Let $T=K[x_{11},x_{21},\dots, x_{v1}]\subset S$ be endowed with the $\ZZ^v$-graded structure induced by that of $S$. Notice that a $\ZZ^v$-graded homogeneous ideal of $T$ is a monomial ideal of $T$.  Hence  a $\ZZ^v$-graded ideal of $S$ which is extended from $T$ is an ideal of $S$ which is generated by monomials in the variables $x_{11},x_{21},\dots, x_{v1}$. We denote by $\HS(M,y)$ the multigraded Hilbert series of a finitely generated $\ZZ^v$-graded $S$-module $M$. 

\begin{defn}
Let $I$ be a $\ZZ^v$-graded ideal of $S$. We say that $I$ is a {\em Cartwright-Sturmfels ideal} if there exists a radical Borel fixed ideal $J$ of $S$ such that $\HS(I,y)=\HS(J,y)$. We say that $I$ is a {\em Cartwright-Sturmfels$^*$ ideal} if there exists a $\ZZ^v$-graded ideal $J$ of $S$ extended from $T$ such that $\HS(I,y)=\HS(J,y)$. 
\end{defn}

In~\cite[Theorem~3.5]{CDG1} we have showed that if $I,J$ are Borel fixed ideals and $\HS(I,y)=\HS(J,y)$, then $I=J$ as soon as $I$ (or $J$) is radical. In particular this implies that 
the multigraded generic initial ideal of a Cartwright-Sturmfels ideal does not depend on the term order (but only on the total order given on the indeterminates). 
 In \cite[Proposition~1.9]{CDG2} we gave a characterization of Cartwright-Sturmfels$^*$  that implies that the multigraded generic initial ideal of a Cartwright-Sturmfels$^*$ ideal is independent of the term order as well. 
Therefore in the sequel the multigraded generic initial ideal of a Cartwright-Sturmfels ideal $I$ will be simply denoted by $\gin(I)$ and similarly for Cartwright-Sturmfels$^*$ ideals. 
Moreover we will sometimes call it generic initial ideal, as we will always deal with the multigraded version.

\section{Maximal minors: the row graded case}\label{rowgrad}

Given integers $m\leq n$, let $S=K[x_{ij} : 1\leq i\leq m, \ \ 1\leq j\leq n]$ with the  $\ZZ^m$-graded structure induced by the assignment  $\deg x_{ij}=e_i\in \ZZ^m$. 
For  $a\in  \NN^m$  let $P_a$ be the associated Borel fixed prime ideal, i.e., 
$$P_a=( x_{ij} :  1\leq i\leq m \mbox{ and } 1\leq j\leq a_i ).$$

Let $L=(\ell_{ij})$ be a row graded $m\times n$ matrix of linear forms, i.e.~the entries of $L$ are  homogeneous and $\deg \ell_{ij}=e_i\in \ZZ^m$. Equivalently, 
 $$\ell_{ij}=\sum_{k=1}^m  \lambda_{ijk} x_{ik}$$ where
$\lambda_{ijk}\in K$. 
Let $I_m(L)$ be the ideal of maximal minors of $L$. 
Under the assumption that $I_m(L)$ has the largest possible codimension   in~\cite[Sect.4]{CDG1} we proved that: 
\begin{thm}
\label{maxcd}
Assume that $I_m(L)$ has codimension $n-m+1$. Then:
\begin{itemize}
\item[(1)] $I_m(L)$ is Cartwright-Sturmfels.
\item[(2)] Its  generic initial ideal is  
$$\gin(I_m(L))=(x_{1a_1}\cdots x_{ma_m} : \sum_{i=1}^m a_i \leq n ).$$
\item[(3)] One has the following irredundant prime decomposition:
$$\gin(I_m(L))=\bigcap_{a\in C} P_a$$ 
where $C=\{  a\in \NN^m :   \sum_{i=1}^m  a_i=n-m+1\}$. 
\end{itemize} 
\end{thm}

In \cite{CDG2} we generalized the first assertion by proving that $I_m(L)$ is Cartwright-Sturmfels independently of its codimension. Our goal here is identifying its  generic initial ideal and the corresponding prime decomposition. 
To this end we introduce the following invariants. For every subset $A$ of $[m]$ let $b_L(A)$ be the dimension of the $K$-vector subspace of $\oplus_{i\in A} S_{e_i}$ generated by the columns of the matrix $L_A=(\ell_{ij})$ with $i\in A$ and $j\in [n]$. Then we have: 

\begin{thm}\label{ginrowgraded}
 With the notation above one has:
\begin{itemize} 
\item[(1)] $I_m(L)$ is Cartwright-Sturmfels.
\item[(2)] The generic initial ideal of $I_m(L)$ is  
$$\gin(I_m(L))=(x_{1a_1}\cdots x_{ma_m} : \sum_{i\in A} a_i  \leq b_L(A)  \mbox{ for every } A\subseteq [m]).$$
\item[(3)] Furthermore
$$\gin(I_m(L))=\bigcap_{a\in C} P_a$$
where  $C$ is the set of the elements $a=(a_1,\dots,a_m)\in\NN^m$ such that  for some  $A\subseteq [m]$  one has $a_i=0$ for  $i\in [m]\setminus A$ and 
$$\sum_{i\in A} a_i= b_L(A)-|A|+1.$$
 \end{itemize} 
\end{thm}

In the proof of the theorem we need the following lemma.
Let $I\subset S$ be a $\ZZ^m$-graded ideal, let $h\in S_{e_k}$ such that the coefficient of $x_{kn_k}$ in $h$ is non-zero. We may identify $S/(h)$ with the polynomial subring $S'$ of $S$  generated by all the variables of $S$ with the exception of $x_{kn_k}$. Then $I+(h)/(h)$ is identified with an ideal $I'$ of $S'$. 

\begin{lemma}\label{quotient}
Let $I$ be a Cartwright-Sturmfels ideal. 
With the notation and identification above one has that  $$\gin(I')S\subseteq\gin(I)$$
for all $h$.
\end{lemma}

\begin{proof}  
In \cite[Theorem~1.16]{CDG2} we have already proved that $I'$ is a Cartwright-Sturmfels ideal of $S'$ and  that $\gin(I')S=\gin(I'S)$.  By construction  $I'$ is the image of $I$ under the $K$-algebra map $\phi_h:S\to S'$ which sends $x_{ij}$ to itself if $(i,j)\neq (k,n_k)$ and $x_{kn_k}$ to $-\lambda^{-1}\sum_{j=1}^{n_k-1} \mu_j x_{kj}$ where  $h=\lambda x_{kn_k}+\sum_{j=1}^{n_k-1} \mu_j x_{kj}$. Denote by $\phi^h:S\to S$  the $K$-algebra  automorphism of $S$ which sends $x_{ij}$ to itself if $(i,j)\neq (k,n_k)$ and $x_{kn_k}$ to $-\lambda^{-1}h$. Since by construction $\phi_h=\phi_{x_{kn_k}}\phi^h$ and since $I$ and $\phi^h(I)$ have the same gin, we may assume $h=x_{kn_k}$. Using a revlex order with $x_{kn_k}$ as smallest variable one has $\inid(I')=\inid(I)\cap S'$, hence $\inid(I')S\subseteq \inid(I)$. Computing the  generic initial ideal on both sides one obtains $\gin(\inid(I')S)\subseteq \gin(\inid(I))$. Finally, since $I$ and $I'$ are Cartwright-Sturmfels ideals, then $\gin(\inid(I))=\gin(I)$ and $\gin(\inid(I'))=\gin(I')$. Hence we conclude that $\gin(I')S\subseteq\gin(I)$, as desired. 
\end{proof} 

\begin{rmk} Without the assumption that $I$ is a Cartwright-Sturmfels ideal the statement in Lemma \ref{quotient} does not hold. For example if $I=(x_1x_2, x_1x_3, x_1^2+x_4^2)$ in $K[x_1,x_2,x_3,x_4]$ with the standard $\ZZ$-grading and $h=x_4$, then $I'=(x_1x_2, x_1x_3, x_1^2)$ and $\gin(I')=I'$ since $I'$ is Borel fixed, while $\gin(I)=(x_1^2, x_1x_2, x_2^2, x_1x_3^2)$. Here the gins are computed with respect to the revlex order. 
The argument given above breaks down when we state that $\gin(\inid(I))=\gin(I)$: this is true for Cartwright-Sturmfels ideals and false in general. 
\end{rmk} 

\begin{proof}[Proof of Theorem \ref{ginrowgraded}] 
As said above, (1) has been proved already in \cite{CDG2}.  

As for (2), let 
$$U=(x_{1a_1}\cdots x_{ma_m} : \sum_{i\in A} a_i  \leq b_L(A)  \mbox{ for every } A\subseteq [m]).$$
We start by proving the inclusion $\gin(I_m(L))\subseteq U$. 
For any $A=\{i_1,\dots,i_v\} \subseteq [m]$ one has $I_m(L)\subseteq I_v(L_A)$. Up to column operations, $L_A$ is  equivalent to  a row graded matrix $L_A'$ of size $v\times b_L(A)$. Hence $\gin(I_m(L))\subseteq\gin(I_v(L'_A))$. Since $L'_A$ can be seen as a multigraded linear section of a matrix $Y_A$ of variables of the same size, applying Lemma~\ref{quotient} and Theorem~\ref{maxcd} we have
$$\gin(I_m(L))\subseteq \gin(I_v(L'_A)) \subseteq \gin(I_v(Y_A))=(x_{i_1a_1}\cdots x_{i_va_v} : a_1+\dots +a_v\leq b_L(A)).$$
Therefore, 
$$\gin(I_m(L))\subseteq\bigcap_{A\subseteq [m]} \gin(I_{|A|} (L_A))\subseteq \bigcap_{A\subseteq [m]} \gin(I_{|A|}(Y_A))$$
and 
$$\bigcap_{A\subseteq [m]} \gin(I_{|A|}(Y_A))  = \bigcap_{A\subseteq [m], A=\{i_1,\dots, i_v\}}  (  x_{i_1a_1}\cdots x_{i_va_v} : a_1+\dots +a_v\leq b_L(A))=U$$
where the last equality  is a straightforward verification. 

In order to prove the reverse inclusion $\gin(I_m(L))\supseteq U$, let $M=x_{1a_1}\cdots x_{ma_m}$ be a monomial such that $ \sum_{i\in A} a_i  \leq b_L(A)$ for every $A\subseteq [m]$. We will show that $M\in\gin(I_m(L))$. 
Denote by $b_i$ the number $b_L(\{i\})$. By assumption $a_i\leq b_i$ and, up to a multigraded linear transformation 
(that does not affect the gin) we may assume that only the first $b_i$ variables of multidegree $e_i$ are actually used in $L$. 
We argue by induction on $N=\sum_{i=1}^m (b_i-a_i)$. Note that $N\geq 0$ by assumption. 

If $N=0$, i.e., $a_i=b_i$ for all $i$, then, after suitable column operations, the matrix $L$ can be brought in the form 
$$\left(\begin{array}{cccccccccccc}
x_{11}  & x_{12} & \ldots & x_{1a_1} & 0 & \ldots & \ldots & \ldots & \ldots & \ldots & \ldots & 0 \\
0 & \ldots & \ldots & 0 & x_{21} & x_{22} & \ldots & x_{2a_2} & 0 & \ldots & \ldots & 0 \\
\vdots & & & & & & & & & & & \vdots \\
0 & \ldots & \ldots & \ldots & \ldots & \ldots & \ldots & 0 & x_{m1} & x_{m2} & \ldots & x_{ma_m}
\end{array}\right).$$
Therefore $I_m(L)=\prod_{i=1}^m(x_{i1}, x_{i2},\ldots, x_{ia_i})$ and $\gin(I_m(L))=I_m(L)$ because $I_m(L)$ is Borel fixed. Hence $M\in I_m(L)=\gin(I_m(L))$ as required.  

Assume now that $N>0$, that is $a_j<b_j$ for some $j$ in $[m]$, say $a_1<b_1$. Let $h$ be a generic linear combination of $\ell_{1j}$ with $j=1,\dots,n$.   By Lemma~\ref{quotient} it suffices to show that $M\in \gin(I_m(L'))$, where $L'$ is the image of $L$ in $S/(h)$. Set $b_i'=b_{L'}(\{i\}$ and notice that, by construction, $b_i'=b_i$ for $i>1$ and $b_1'=b_1-1$. Then we may conclude by induction provided that we show that
\begin{equation}\label{quot_inequalities}
\sum_{i\in A} a_i  \leq b_{L'}(A),\end{equation} 
for all $A\subseteq [m]$ and for $h$ generic.  
As we will see, inequality (\ref{quot_inequalities}) follows essentially from Grassmann's formula. 
To this end let $V_i$ be the kernel of the $K$-linear map 
$f_i:K^n\to S_{e_i}$ defined by $f((\lambda_1,\dots , \lambda_n))=\sum_{j=1}^n \lambda_j L_{ij}$ and let $V_i'$ be the corresponding objects associated to $L'$.  
By construction, $V_i'=V_i$ for $i>1$,  $V_1'=V_1+\langle h'\rangle$ with $h'$ generic. Furthermore
$ b_L(A)=n-\dim_K \cap_{i\in A} V_i$ and $b_L'(A)=n-\dim_K \cap_{i\in A} V_i'$. If $1\not\in A$ then $ b_L(A)=b_L'(A)$ and (\ref{quot_inequalities}) holds by assumption. 
If $1\in A$ let $W=\cap_{i\in A, i\neq 1} V_i$. Then, by Grassmann's formula we have: 
$$b_L(A)=n-\dim W-\dim V_1+\dim (V_1+W)$$ 
and 
$$b_L'(A)=n-\dim W-\dim (V_1+\langle h'\rangle)+\dim (V_1+ \langle h'\rangle+W)$$ 
Now, if $V_1+W$ is strictly contained in $K^n$, then $\dim (V_1+\langle h'\rangle)=\dim V_1+1$ and $\dim (V_1+ h'\rangle+W)=\dim (V_1+W)+1$ for a generic $h'$. Hence it follows that $b_L(A)=b_L'(A)$ and (\ref{quot_inequalities}) holds by assumption. 

Finally if $V_1+W=K^n$ then we have
$$b_L'(A)= n-\dim W-\dim (V_1+\langle h' \rangle)+n=b_L(A\setminus \{1\})+b_1-1$$ 
and then 
$$\sum_{i\in A} a_i=a_1+\sum_{i\in A, i\neq 1} a_i<b_1+b_L(A\setminus \{1\})=b_{L'}(A)+1$$
that is, 
$$\sum_{i\in A} a_i\leq b_{L'}(A)$$
as desired. 

We now prove (3). In the proof of (2) we have shown that 
$$\gin(I_m(L))=\bigcap_{A\subseteq [m]} \gin(I_{|A|}(Y_A)).$$
Moreover, by Theorem~\ref{maxcd} $$\gin(I_{|A|}(Y_A))=\bigcap_{a\in C_A} P_a$$
where $C_A$ is the set of the $a\in \ZZ^m$ with $a_i=0$ for $i\in [m]\setminus A$ and $\sum_{i\in A} a_i=b_L(A)-|A|+1$. Combining the two equalities we get the desired prime decomposition. 
\end{proof}

\section{Maximal minors: the column graded case}\label{colgrad}

For completeness, in this section we recall the results proved in \cite{CDG1} concerning ideals of maximal minors in the column graded case.  
Given integers $m\leq n$, let $S=K[x_{ij} : 1\leq i\leq m, \ 1\leq j\leq n]$ graded by $\deg x_{ij}=e_j\in \ZZ^n$. 
For  $a\in \NN^n$  with $a_i\leq m$ let $P_a$ be the associated Borel fixed prime ideal, i.e., 
$$P_a=( x_{ij} :  1\leq i\leq a_i  \mbox{ and } 1\leq j\leq n ).$$
Let $L=(\ell_{ij})$ be a column graded $m\times n$ matrix of linear forms, i.e.  the  $\ell_{ij}$'s are homogeneous of  degree $e_j$. Equivalently, 
$$\ell_{ij}=\sum_{k=1}^m  \lambda_{ijk} x_{kj}$$ where
$\lambda_{ijk}\in K$.  

Let $I_m(L)$ be the ideal of maximal minors of $L$. In \cite[Section 3]{CDG1} we have proved  the following statement. We rephrase the result using the terminology introduced in \cite{CDG2}.

\begin{thm}
\label{superBSZ2} 
With the notation above one has:
\begin{itemize} 
\item[(1)] $I_m(L)$ is a Cartwright-Sturmfels ideal as well as a Cartwright-Sturmfels$^*$ ideal. 
\item[(2)] The generic initial ideal of $I_m(L)$ is
$$\gin(I_m(L))=(x_{1a_1}\cdots x_{1a_m} : [a_1,\dots, a_m]_L\neq 0 )$$ where $[a_1,\dots, a_m]_L$ denotes the $m$-minor of $L$ corresponding to columns indices $a_1,\dots,a_m$.
\item[(3)] The prime decomposition of $\gin(I_m(L))$ is 
$$\gin(I_m(L))=\bigcap_{c\in C} P_c$$
where $C$ is the set of the elements $c=(c_1,\dots,c_n)\in \{0,1\}^n$ whose support $\Supp(c)=\{ i \in [n] : c_i=1\}$ is minimal with respect to the property that $\Supp(c)\cap\{a_1,\ldots,a_m\}\neq \emptyset$ for every $(a_1,\ldots,a_m)$ such that $[a_1,\dots, a_m]_L\neq 0$. In other words, the minimal associated primes of $I_m(L)$ are exactly the ideals $(x_{1b_1},\ldots,x_{1b_k})$ where $\{b_1,\ldots,b_k\}$ is a minimal vertex cover of the simplicial complex whose facets are $\{\{a_1,\ldots,a_m\}\mid [a_1,\dots, a_m]_L\neq 0\}$. 
\end{itemize} 
\end{thm} 

\begin{rmk}
If all the minors of $L$ are nonzero, then $$\gin(I_m(L))=(x_{1a_1}\cdots x_{1a_m} : 1\leq a_1<\ldots<a_m\leq n)=\bigcap_{c\in C} P_c$$ where $C$ is the set of all subsets of $[n]$ of cardinality $n-m+1$. Notice moreover that $$\gin(I_m(L))=\gin(I_m(X_{m\times n}))=\gin(I_{n-m+1}(X_{(n-m+1)\times n}))^*$$ where $X_{u\times v}$ denotes a matrix of indeterminates of size $u\times v$ and $*$ denotes the Alexander dual.
\end{rmk}

\section{A family of Hilbert series of multigraded algebras }

In this section we discuss a combinatorially defined family of formal power series. We will show that every such power series is the  Hilbert series of a multigraded $K$-algebra defined by a radical Borel fixed ideal. In Section~\ref{2min} we will apply these results to ideals of $2$-minors of multigraded matrices.  

Let $n,m\in \NN_{>0}$ and let $$\Phi:2^{[m]}\to \{0,\dots,n\}$$ be a map such that: 
\begin{itemize}
\item[(1)] $\Phi(\emptyset)=n$, 
\item[(2)] $\Phi(A)\geq \Phi(B)$ whenever $A\subseteq B$. 
\end{itemize} 
Consider the formal power series
$$H_\Phi(y_1,\dots,y_m)= \sum_{A\subseteq [m]} \sum_{a\in \NN^m \atop  \Supp(a)=A} \binom{\Phi(A)-1+|a|}{\Phi(A)-1} y^a \in \QQ[|y_1,\dots,y_m|]$$
where for $a=(a_1,\dots,a_m)\in \NN^m$ one sets 
$$\Supp(a)=\{ i\in [m] : a_i>0\} \quad \mbox{ and } \quad  |a|=a_1+\dots+a_m.$$

We now prove that the series $H_\Phi(y_1,\dots,y_m)$ is the multigraded Hilbert series of $S/I$, where $S$ is a multigraded polynomial ring and $I$ is a radical Borel fixed ideal. 
Let $d_i=\Phi(\{i\})$ and let $S=K[x_{ij} : i=1,\dots, m \mbox{ and } 1\leq j\leq d_i]$ endowed with the $\ZZ^m$-multigraded structure induced by 
$\deg x_{ij}=e_i\in \ZZ^m$. For every $\emptyset\neq A=\{a_1,\dots, a_t\} \subseteq  [m]$ let
$$I_A=( \prod_{i\in A} x_{ib_i}  :  1\leq b_i\leq d_{i} \mbox{ for every } i\in A   \mbox{ and }  \sum_{i\in A} (d_i-b_i) \geq  \Phi(A) ).$$

\begin{lemma}\label{lemac2}
The ideal $I=\sum_{A\subseteq  [m]} I_A$ is radical Borel fixed and $S/I$ has multigraded Hilbert series equal to $H_\Phi(y_1,\dots,y_m)$. 
\end{lemma} 

\begin{proof} 
It is clear that each $I_A$ is radical and Borel fixed. It then follows that also $I$ is radical and Borel fixed. 
It remans to prove that $S/I$ has multigraded Hilbert series equal to $H_\Phi(y_1,\dots,y_m)$. In other words,  we have to prove that the number of monomials of $S$ not in $I$ and of multidegree $a\in \NN^m$ is exactly $\binom{\Phi(A)-1+|a|}{\Phi(A)-1}$ where $A=\Supp(a)$.  Whenever $A\subseteq B$  the homogeneous components of $I_A$ of multidegrees bigger than or equal to $\sum_{i\in B} e_i$ are already  contained  in $I_B$ because, by assumption, $\Phi(A)\geq \Phi(B)$. 
Therefore it suffices to prove that for every $a\in \NN^m$ and $A=\Supp(a)$ the number of monomials not in $I_A$ and of multidegree $a$ is exactly $\binom{\Phi(A)-1+|a|}{\Phi(A)-1}$. 
Without loss of generality we may also assume that $A=\{1,\dots,m\}$ and then set $J=I_A$ and $u=\Phi(A)$. Then 
$$J=( x_{1b_1}\cdots x_{mb_m} : 1\leq b_i\leq  d_i \mbox{ and } b_1+\cdots +b_m\leq d_1+\dots +d_m-u)$$
and we have to prove that the number of monomials not in $J$ and of multidegree $a\in \NN^m$  and full support  $\Supp(a)=\{1,\dots, m\}$ is exactly $\binom{u-1+|a|}{u-1}$. 
Let $c=(c_1,\dots,c_m)\in \NN^m$ with $1\leq c_i\leq d_i$ such that 
\begin{equation}
\label{ac1} 
\sum_{i=1}^m c_i>\sum_{i=1}^m d_i-u.
\end{equation} 
For such a vector  $c$ and given $a\in \NN_{>0}^m$ we consider the set $X_c$ of monomials $M$ of multidegree $a$ such that for every $i$ we have $\min\{ j :  x_{ij}|M \}=c_i$. By construction, the set of the monomials of multidegree $a$ that are not in $J$ is the disjoint union of the $X_c$'s where $c$ runs in the set of vectors specified above. Since $$\# X_c=\prod_{i=1}^m  \binom{d_i-c_i+a_i-1}{a_i-1}$$ it is enough to prove the following identity: 
\begin{equation}
\label{ac2} 
\sum_{c} \prod_{i=1}^m  \binom{d_i-c_i+a_i-1}{a_i-1}=\binom{u-1+|a|}{u-1}
\end{equation} 
where the sum runs over all the vectors $c$ described above.  
Setting $w_i=d_i-c_i$ we have to prove that 
\begin{equation}
\label{ac4}
\sum_{w} \prod_{i=1}^m  \binom{w_i+a_i-1}{a_i-1}=\binom{u-1+|a|}{u-1}
\end{equation} 
where the sum runs over all the vectors $w\in \NN^m$ such that $0\leq w_i\leq d_i-1$ and 
\begin{equation}
\label{ac5}\sum_{i=1}^m w_i\leq u-1.
\end{equation}
Since, by construction,  $u\leq d_i$ for every $i$, the restriction $w_i\leq d_i-1$ is subsumed by (\ref{ac1}). Now the identity (\ref{ac4}) is easy: Both the left and the right side count the number of monomials of total degree $\leq u-1$ in a set of variables which is a disjoint union of subsets of cardinality $a_1,a_2,\dots,a_m$. 
\end{proof}

One can also identify the prime decomposition of the ideal described in Lemma \ref{lemac2}. 
For $v=(v_1,\dots,v_m)\in \NN^m$  with $v_i\leq d_i$ let $P_v$ be the prime Borel fixed ideal associated to $v$, that is, 
$$P_v=(x_{ij} : 1\leq i\leq m \mbox{ and } 1\leq j\leq v_i ).$$

\begin{lemma}
\label{lemac3a}  Let  $I$ be the ideal of Lemma \ref{lemac2}. For every $v\in \NN^m$ with $v_i\leq d_i$  set $A(v)=\{ i : v_i<d_i\}$.  The following are equivalent: 
\begin{itemize}
\item[(1)] $P_v\supseteq I$,  
\item[(2)] $\displaystyle \sum_{i=1}^m v_i \geq \sum_{i=1}^m d_i-\Phi(A(v))-|A(v)|+1$. 
\end{itemize} 
\end{lemma} 

\begin{proof} 
For the implication  $(1) \Rightarrow (2)$ we assume, by contradiction, that there exists a $v$ such that $P_v\supseteq I$ and $v$ does not satisfy (2). Setting $A=A(v)$, one has that 
$$\sum_{i\in A} v_i\leq  \sum_{i\in A} d_i-\Phi(A)-|A|.$$ Now let $b_i=v_i+1$ for every $i\in A$. By construction, $b_i\leq d_i$ and 
$$\sum_{i\in A} b_i=\sum_{i\in A} v_i+|A|\leq \sum_{i\in A} d_i-\Phi(A).$$
Hence $\prod_{i\in A} x_{ib_i}\in I$ and, by construction, $\prod_{i\in A} x_{ib_i}\not\in P_v$, contradicting the assumption. 

To prove the converse, assume that $v$ satisfies (2) and let $M=\prod_{i\in A} x_{ib_i}$ be a generator of $I$ with $A\subset [m]$. Let $B=A(v)$. 
If there exists $i\in A\setminus B$, then $v_i=d_i$ and hence $x_{ib_i}\in P_v$, that is,  $M\in P_v$. If instead $A\subseteq  B$ and  $b_i>v_i$ for every  $i\in A$ we have: 
$$\begin{array}{ll}
\displaystyle \sum_{i\in A} d_i-\Phi(A)\geq  \sum_{i\in A} b_i\geq  \sum_{i\in A} v_i +|A|=\sum_{i\in B} v_i  -\sum_{i\in B\setminus  A} v_i +|A| \\
\displaystyle  \geq \sum_{i\in B} d_i -\Phi(B)-|B|+1+|A|-\sum_{i\in B\setminus  A} v_i 
\end{array}$$
Hence 
$$\begin{array}{ll}
\displaystyle  -\Phi(A)\geq \sum_{i\in B\setminus A } (d_i-v_i)  -\Phi(B)-|B|+1+|A| \geq\\ 
\displaystyle   \geq  | B\setminus A|  -\Phi(B)-|B|+1+|A|=  -\Phi(B)+1
\end{array}$$
that is, 
$$\Phi(B)\geq \Phi(A)+1$$
a contradiction.  Therefore there exists an index $i\in A$ such that $v_i\geq b_i$. This implies that the monomial $M$ is in $P_v$. 
\end{proof} 

\begin{lemma}\label{lemac3} 
The ideal $I$ of Lemma \ref{lemac2} has the following  irredundant   prime decomposition:
$$I=    \cap_{v}   P_v$$
where $v=(v_1,\dots,v_m)$ varies among the vectors of $\NN^m$ such that: 
\begin{itemize}
\item[(1)]  $0\leq v_i \leq d_i$ for every $i=1,\dots,m$, 
\item[(2)]  $\displaystyle \sum_{i=1}^m v_i= \sum_{i=1}^m d_i-\Phi(A(v))-|A(v)|+1$ where $A(v)=\{ i : v_i<d_i\}$,
\item[(3)] $\Phi(B)<\Phi(A(v))$ for every $B\supsetneq A(v)$.
\end{itemize} 
\end{lemma}  

\begin{proof} Since $I$ is radical and Borel fixed, it is the intersection of the $P_v$'s that contain it. They are described in Lemma \ref{lemac3a} and we have to prove that, among them, those that are minimal over $I$ are exactly those described by conditions (2) and (3). First observe that if $v$ verifies 
\begin{equation}\label{ac23}
\sum_{i=1}^m v_i> \sum_{i=1}^m d_i-\Phi(A(v))-|A(v)|+1
\end{equation} 
then $P_v$ is not minimal over $I$. 
To this end we set  $w=v-e_i$ where $i\in A(V)$ and $v_i>0$. Such an $i$ exists, since otherwise from (\ref{ac23}) one would deduce that 
$$0> \sum_{i\in A(V)} d_i-\Phi(A(v))-|A(v)|+1$$
which is a contradiction because $d_i>0$ for every $i\in A(v)$.
For such a $w$ we have $A(w)=A(v)$ and hence  
 $$ \sum_{i=1}^m w_i \geq \sum_{i=1}^m d_i-\Phi(A(w))-|A(w)|+1$$
 that implies $I\subseteq P_w\subsetneq P_v$. 
 
Secondly we show that if $v$ satisfies (2) but not (3) then  $P_v$ is not minimal over $I$. By assumption there exists $i\in [m]$, $i\not\in A(v)$, such that if we set $B=A(v)\cup\{i\}$ we have $\Phi(B)=\Phi(A(v))$. If $d_i>0$ the we may set $w=v-e_i$ so that $B=A(w)$ . Then one  checks that $w$ satisfies (1) and hence $I\subseteq P_w\subsetneq P_v$. If instead $d_i=0$ then $\Phi(B)=0$ because $i\in B$ and hence $\Phi(A(v))=0$. Then by (2)
 $$|A(v)|=\sum_{i\in A(v)} (d_i-v_i)+1$$ 
 which is a a contradiction since $d_i>v_i$ for every $i\in A(v)$. 
 
Finally we have to check that every $v$ satisfying (2) and (3) corresponds to a prime which is minimal over $I$. By contradiction, let $v$ satisfy (2) and (3), and assume that there exists a $w$ satisfying (2) such that $w\leq v$ componentwise and $w_j<v_j$ for some $j\in [m]$. Then $A(v)\subseteq A(w)$ and $j\in A(w)$. If  $A(v)=A(w)$ then 
$$\sum_{i=1}^m d_i-\Phi(A(w))-|A(w)|-1 = \sum_{i=1}^m w_i<\sum_{i=1}^m v_i=\sum_{i=1}^m d_i-\Phi(A(v))-|A(v)|-1,$$ 
a contradiction because the first and last expression are equal. If instead $A(v)\subsetneq A(w)$ then 
$$\sum_{i=1}^m (v_i-w_i)=|A(w)|-|A(v)|+\Phi(A(w))-\Phi(A(v))$$
because both $v$ and $w$ satisfy (2). On the other hand, 
$$\sum_{i=1}^m (v_i-w_i)\geq |A(w)|-|A(v)|$$
because for every $j\in A(w)\setminus A(v)$ one has $d_j=v_j>w_j$. Summing up, 
$$\Phi(A(w))-\Phi(A(v))\geq 0$$
that is, 
$$\Phi(A(w))=\Phi(A(v))$$
contradicting the assumption that $v$ satisfies (3). 
\end{proof}

\section{Ideals generated by $2$-minors}\label{2min}

We now apply Lemmas 
 \ref{lemac2} and~\ref{lemac3} to obtain a description of the generic initial ideal for ideals of $2$-minors of multigraded matrices. Up to a transposition of the matrix, it is not restrictive to assume that the matrix is row graded. 

For integers $m,n$ let $S=K[x_{ij} : i=1,\dots,m \mbox{ and } j=1,\dots, n]$ be endowed with the multigrading induced by $\deg x_{ij}=e_i\in \ZZ^m$. Let $L=(\ell_{ij})$ be a $m\times n$ matrix of linear forms with  $\deg \ell_{ij}=e_i$.  We know by \cite{CDG2} that the ideal $I_2(L)$ of $2$-minors of $L$ is a Cartwright-Sturmfels ideal and we want to describe  its generic initial ideal. 

Consider the surjective map of $K$-algebras 
\begin{equation}\label{ac6}
\psi: K[x_{ij}]\to K[\ell_{ij}] \mbox{ defined by } \psi(x_{ij})=\ell_{ij} 
\end{equation} 
Its kernel is of the form $\sum_{i=1}^m W_i$, where $W_i$ is the  $K$-subspace of $S_{e_i}$ corresponding to the $K$-linear dependence relations on the $i$-th row of $L$. Going modulo the $2$-minors on each side of the map (\ref{ac6}) we get an isomorphism
\begin{equation}\label{ac7}
K[x_{ij}]/I_2(X)+\Ker(\psi) \simeq  K[\ell_{ij}]/I_2(L) 
\end{equation} 
where $X=(x_{ij})$. Identifying $K[x_{ij}]/I_2(X)$ with the Segre product 
$$R=K[x_iy_j : i=1,\dots,m \mbox{ and } j=1,\dots, n]$$  
via the map sending $x_{ij}$ to $x_iy_j$, we obtain an isomorphism 
\begin{equation}\label{ac8}
R/J \simeq  K[\ell_{ij}]/I_2(L)
\end{equation} 
where $J=\sum_{i=1}^m (W'_ix_i)$ where each $W'_i$ is  a space of linear forms in $y_1,\dots, y_n$. 
Now for a subset $A$ of $[m]$ we set
$$\phi(A)=n-\dim_K \sum_{a\in A} W'_a$$
Since the ideal $J$ is monomial in the variables $x_i$, it is easy to compute the multigraded Hilbert series of the quotient it defines. It turns out that the multigraded Hilbert series of $R/J$, and hence of $K[\ell_{ij}]/I_2(L)$, is given by 
$$\sum_{A\subseteq [m]} \sum_{a\in \ZZ^m \atop A=\Supp(a) } \binom{ \phi(A)-1+|a|}{\phi(A)-1} y^a.$$
Hence, by Lemmas~ 
\ref{lemac2} and~\ref{lemac3} we conclude

\begin{thm}\label{2minthm} 
With the notation above one has:
\begin{itemize} 
\item[(1)] $I_2(L)$ is a Cartwright-Sturmfels ideal.
\item[(2)] The  generic initial ideal of $I_2(L)$  is generated by the monomials of the form 
$\prod_{i\in A} x_{ib_i}$ such that 
$$\begin{array}{ll}
A\subseteq  [m], \; \; 1\leq b_i\leq n-\dim_K V_{i} \mbox{ for every } i\in A, \mbox{ and } \\  \\ 
\displaystyle \sum_{i\in A} b_i \leq  n(|A|-1)+\dim_K V_A- \sum_{i\in A} \dim_K V_{i} 
\end{array}$$
where $V_i=\{ \lambda\in K^n  | \sum_{i=1}^n  \lambda_j L_{ij}=0\}$ and $V_A=\sum_{i\in A} V_i$. 
\item[(3)] The irredundant  prime decomposition of  $\gin(I_2(L))$ is given by 
$$\gin(I_2(L))=\cap_{v}P_v$$ 
where $v$ varies among the vectors $v\in \NN^m$ satisfying the following conditions:

$$ 0\leq v_i \leq n-\dim_K V_i \mbox{ for every } i=1,\dots,m, $$
$$ \sum_{i=1}^m v_i=n(m-1)+1-|A|+\dim_K V_A- \sum_{i\in A} \dim_K V_{i} \mbox{  where }A=\{ i : v_i<n-\dim_K V_i \},$$
$$\dim_K V_B>\dim_K V_A \mbox{  for every }B\supsetneq A.$$

\end{itemize} 
\end{thm} 

\begin{rmk} 
Notice that the argument above gives another proof of the fact that the ideal of $2$-minors of a multigraded matrix is Cartwright-Sturmfels. 
\end{rmk} 

\begin{rmk}  
In the case of non-maximal minors of size at least $3$, the corresponding ideals are not Cartwright-Sturmfels. 

For example, let $X$ be a matrix of variables of size $4\times 4$ and let $\tau$ be the revlex order. Then $\gin_{\tau}(I_3(X))$ has $x_{1,1}^2x_{2,2}x_{3,2}x_{4,2}$ 
among its minimal generators. 

In addition, let $X$ be a matrix of variables of size $4\times 5$ and let $\tau$ be the lex order associated to the following ordering of the variables:
$$x_{11},x _{22},x _{33},x _{44},
x _{12},x _{23},x _{34}, x_{45}, 
x _{21},x _{32},x _{43},
x _{13},x _{24}, x_{35},
x _{31},x _{42},
x_{14},x_{25},
x _{41},
x_{15}
.$$
Then $\inid(I_3(X))$ has $x_{12}x_{23}x_{31}x_{45}^2$ and $ x_{12}x_{23}x_{31}x_{44}^2$ among its minimal generators and it does not define a Cohen-Macaulay ring. 

Nevertheless it would be interesting to compute the multigraded generic initial ideal in these cases as well. 
\end{rmk}

\end{document}